\documentclass[12pt]{amsart}

\usepackage{amsmath}
\usepackage{amscd}
\usepackage{eucal}
\usepackage{amssymb}
\usepackage{epic}
\usepackage{graphics}
\usepackage{color}
\usepackage{mathrsfs}
\usepackage{multirow}
\usepackage{subfigure}
\usepackage{graphicx}
\usepackage{mathtools}

\numberwithin{equation}{section}
\numberwithin{table}{section}

\theoremstyle{plain}
\newtheorem{theorem}{Theorem}[section]
\newtheorem*{thm:A}{Theorem \ref{thm:A}}
\newtheorem*{thm:B}{Theorem \ref{thm:B}}
\newtheorem{lemma}{Lemma}[section]

\theoremstyle{remark}
\newtheorem*{remark}{Remark}

\newcommand{\BB}{{\mathcal{B}}}
\newcommand{\CC}{{\mathcal{C}}}

\newcommand{\Zp}{{\mathbb{Z}_{p}}}
\newcommand{\Qp}{{\mathbb{Q}_{p}}}
\newcommand{\Fp}{{\mathbb{F}_{p}}}

\newcommand{\ZZ}{{\mathbb{Z}}}
\newcommand{\QQ}{{\mathbb{Q}}}
\newcommand{\FF}{{\mathbb{F}}}
\newcommand{\allone}{{\mathbf{1}}}

\DeclareMathOperator{\im}{Im}

\DeclareMathOperator{\coker}{coker}

\title[Critical group of a Moore(57, 2)]{On the Critical Group of the missing Moore graph.}
\author[Ducey]{Joshua E. Ducey}

\address{Dept.\ of Mathematics and Statistics, James Madison University, Harrisonburg, VA 22807}
\email{duceyje@jmu.edu}

\keywords{invariant factors, elementary divisors, Smith normal form, critical group, Jacobian group, sandpile group, adjacency matrix, Laplacian, chip-firing, Moore graph}
\subjclass[2010]{05C50}

\begin{document}
\begin{abstract}
We consider the critical group of a hypothetical Moore graph of diameter $2$ and valency $57$.  Determining this group is equivalent to finding the Smith normal form of the Laplacian matrix of such a graph.  We show that all of the Sylow $p$-subgroups of the critical group must be elementary abelian with the exception of $p = 5$.  We prove that the $5$-rank of the Laplacian matrix determines the critical group up to two possibilities.

\end{abstract}
\maketitle
\section{Introduction}
Consider a simple graph with diameter $d$ and girth $2d+1$.  Such a graph is necessarily regular, and is known as a \textit{Moore graph}.  Another characterization: Moore graphs are the regular graphs of diameter $d$ and valency $k$ that achieve the upper bound on number of vertices
\[
1+\sum_{i = 1}^{d} k \cdot (k-1)^{i-1}.
\]
We will denote such a graph of diameter $d$ and valency $k$ as a Moore$(k,d)$.
It was shown in \cite{hs} that for Moore graphs of diameter $2$, one must have the valency $k \in \{2, 3, 7, 57\}$.  The $5$-cycle, the Petersen graph, and the Hoffman-Singleton graph are the unique graphs satisfying the first three respective degrees.  Neither the existence nor uniqueness of a Moore graph of diameter $2$ and valency $57$ have yet been established.  

There has been some work on determining algebraic properties of such a graph, especially regarding its automorphism group \cite{aschbacher, macaj-siran}.  It is known that a Moore$(57,2)$ possesses very few automorphisms, if any at all.  For a more recent result on the enumeration of independent sets in such a graph, see \cite[Theorem 5.1]{alexander-mink}.  

In this paper we investigate the structure of the \textit{critical group} of a Moore($57,2$).  We define this abelian group formally in the next section, but we mention here that it is an important graph invariant that has been widely studied and goes by many names in the literature (sandpile group, Jacobian group, Picard group).  The group comes from the Laplacian matrix of the graph and has order equal to the number of spanning trees of the graph.  The critical group can also be understood in terms of a certain ``chip-firing'' game on the vertices of the graph \cite{biggs}, \cite[Chap. 14]{godsil-royle}.

In Section \ref{sec:prelims} we give formal definitions, and describe the relation between the critical group and the Laplacian matrix of a graph.  Our main results are Theorems \ref{thm:A} and \ref{thm:B}, which together show that the $5$-rank of the Laplacian matrix of a Moore$(57,2)$ determines the critical group to within two possibilities.  We state these theorems immediately below for the interested reader.  They will be proved in Section \ref{sec:proofs}.  The critical group of a graph $\Gamma$ is denoted $K(\Gamma)$.  Let $Syl_{p}(K(\Gamma))$ denote the Sylow $p$-subgroup of the critical group.

\vspace{\baselineskip}

\begin{thm:A}
Let $\Gamma$ denote a Moore$(57,2)$ graph.  Then for some nonnegative integers $e_{1}, e_{2}, e_{3}$ we have
\[
K(\Gamma) \cong \left( \ZZ / 2\ZZ\right)^{1728} \oplus \left( \ZZ / 13\ZZ\right)^{1519} \oplus \left(\ZZ / 5 \ZZ\right)^{e_{1}} \oplus \left(\ZZ / 5^{2} \ZZ\right)^{e_{2}} \oplus \left(\ZZ / 5^{3} \ZZ\right)^{e_{3}}.
\]
\end{thm:A}

\vspace{\baselineskip}

\begin{thm:B}
Let $\Gamma$ be a Moore$(57,2)$ graph.  Let $e_{0}$ denote the rank of the Laplacian matrix of $\Gamma$ over a field of characteristic $5$.  Then either
\[
Syl_{5}(K(\Gamma)) \cong \left(\ZZ / 5 \ZZ\right)^{1520-e_{0}} \oplus \left(\ZZ / 5^{2} \ZZ\right)^{1732-e_{0}} \oplus \left(\ZZ / 5^{3} \ZZ\right)^{e_{0}-3}
\]
or
\[
Syl_{5}(K(\Gamma)) \cong \left(\ZZ / 5 \ZZ\right)^{1521-e_{0}} \oplus \left(\ZZ / 5^{2} \ZZ\right)^{1730-e_{0}} \oplus \left(\ZZ / 5^{3} \ZZ\right)^{e_{0}-2}.
\]
\end{thm:B}

\vspace{\baselineskip}

\section{Preliminaries}\label{sec:prelims}
Let $\Gamma$ be a simple graph with some fixed ordering of the vertex set $V(\Gamma)$.  Then the \textit{adjacency matrix} of $\Gamma$ is a square matrix $A = (a_{i,j})$ with rows and columns indexed by $V(\Gamma)$, where
\[
a_{i,j} = \begin{cases}
1, &\mbox{ if vertex $i$ is adjacent to vertex $j$,}\\
0, &\mbox{ otherwise}.
\end{cases}
\]
Let $D = (d_{i,j})$ be a matrix of the same dimensions as $A$ with
\[
d_{i,j} = \begin{cases}
\mbox{the degree of vertex $i$,} &\mbox{ if $i = j$,}\\
0, &\mbox{ otherwise}.
\end{cases}
\]
Finally, set $L = D - A$.  The matrix $L$ is called the \textit{Laplacian matrix} of the graph $\Gamma$, and will be our primary focus.  

Let $\ZZ^{V(\Gamma)}$ denote the free abelian group on the vertex set of $\Gamma$.  Then the Laplacian $L$ can be understood as describing a homomorphism:
\[
L \colon \ZZ^{V(\Gamma)} \to \ZZ^{V(\Gamma)}.
\]
We will usually use the same symbol for both the matrix and the map.  The cokernel of $L$, 
\[
\coker L = \ZZ^{V(\Gamma)} / \im(L),
\]
 always has free rank equal to the number of connected components of $\Gamma$.  The torsion subgroup of $\coker L$ is known as the \textit{critical group} of $\Gamma$, and is denoted $K(\Gamma)$.  It is an interesting fact that for a connected graph $\Gamma$, the order of $K(\Gamma)$ is equal to the number of spanning trees of $\Gamma$.  See \cite{biggs} or \cite{dino1} for proofs of these basic facts and more information.  One way to compute the critical group of a graph is by finding the \textit{Smith normal form} of $L$.  

Recall that if $M$ is any $m \times n$ integer matrix then one can find square, unimodular (i.e., unit determinant) matrices $P$ and $Q$ so that $PMQ = S$, where the matrix $S=(s_{i,j})$ satisfies:
\begin{enumerate}
\item $s_{i,i}$ divides $s_{i+1,i+1}$ for $1 \leq i < \min \{m,n\}$
\item $s_{i,j} = 0$ for $i \neq j$.
\end{enumerate}
Then $S$ is known as the Smith normal form of $M$, and it is not hard to see that
\[
\coker M \cong \ZZ/s_{1,1}\ZZ \oplus \ZZ/s_{2,2}\ZZ \oplus \cdots
\]
This particular decomposition of $\coker M$ is the invariant factor decomposition, and the integers $s_{i,i}$ are known as the \textit{invariant factors} of $M$.  The prime power factors of the invariant factors of $M$ are known as the \textit{elementary divisors} of $M$.

The concept of Smith normal form generalizes nicely when one replaces the integers with any principal ideal domain (PID), as is well known (see, for example, \cite[Chap. 12]{d-f}).  In what follows $J$ and $I$ will be used to denote the all-ones matrix and the identity matrix, respectively, of the correct sizes. 

\section{The Critical group of a Moore($57, 2$)} \label{sec:proofs}
Throughout the rest of the paper we let $\Gamma$ denote a Moore($57, 2$) graph.  It follows easily from the definitions that $\Gamma$ is strongly regular with parameters
\[
v = 3250,
k = 57,
\lambda = 0,
\mu = 1
\]
and so the adjacency matrix $A$ must satisfy
\[
A^2 = 57I + 0A + 1(J-A-I)
\]
or
\begin{equation}\label{eqn:A}
A^2 = 56I - A + J.
\end{equation}
From this equation one can deduce \cite[Chap. 9]{b-h} that $A$ has eigenvalues $7, -8, 57$ with respective multiplicities $1729, 1520, 1$.  The degree $57$ has eigenvector the all-one vector $\allone$; the other eigenvalues are the \textit{restricted} eigenvalues.  

Since the graph is regular, we immediately get the Laplacian spectrum:  eigenvalues $50, 65, 0$ with multiplicities as above.  Kirchhoff's Matrix-Tree Theorem \cite[Prop. 1.3.4]{b-h} tells us that the number of spanning trees of $\Gamma$ is the product of the non-zero eigenvalues, divided by the number of vertices.  We thus get the order of the critical group of $\Gamma$:
\begin{align*}
|K(\Gamma)| &= \frac{1}{3250} \cdot 50^{1729} \cdot 65^{1520}\\
&= 2^{1728} \cdot 5^{4975} \cdot 13^{1519}.
\end{align*}

We remark that the number of such abelian groups is quite large.  The next theorem begins to narrow things down.  Let $Syl_{p}(K(\Gamma))$ denote the Sylow $p$-subgroup of the critical group.

\begin{theorem}\label{thm:A}
Let $\Gamma$ denote a Moore$(57,2)$ graph.  Then for some nonnegative integers $e_{1}, e_{2}, e_{3}$ we have
\[
K(\Gamma) \cong \left( \ZZ / 2\ZZ\right)^{1728} \oplus \left( \ZZ / 13\ZZ\right)^{1519} \oplus \left(\ZZ / 5 \ZZ\right)^{e_{1}} \oplus \left(\ZZ / 5^{2} \ZZ\right)^{e_{2}} \oplus \left(\ZZ / 5^{3} \ZZ\right)^{e_{3}}.
\]
\end{theorem}
\begin{proof} \hfil \\
Substituting $A = 57I - L$ into equation \ref{eqn:A}, we get
\begin{align}
(57I - L)^{2} &= 56I - (57I - L) + J \nonumber \\
L^{2} - 115L &= -3250I + J \nonumber \\
(L - 115I)L &= -(2\cdot 5^{3} \cdot 13)I + J. \label{eqn:L}
\end{align}
This last equation tells us much about the Smith normal form of $L$.  As in the previous section, we view $L$ as defining a homomorphism of free $\ZZ$-modules
\[
L \colon \ZZ^{V(\Gamma)} \to \ZZ^{V(\Gamma)}.
\]
Define a subgroup of $\ZZ^{V(\Gamma)}$:
\[
Y = \left\{ \sum_{v \in V(\Gamma)} a_{v}v \, \,\Big\vert \, \sum_{v \in V(\Gamma)} a_{v} = 0 \right\}.
\]
Note that $Y$ is the smallest direct summand of $\ZZ^{V(\Gamma)}$ that contains $\im L$ (i.e., it is the \textit{purification} of $\im L$).  Changing the codomain of $L$ to $Y$ does not affect the nonzero invariant factors of $L$, so we do.  In fact, with this adjustment we have $\coker L \cong K(\Gamma)$.

If we also restrict the domain of $L$ to $Y$ the Smith normal form will probably be altered.  However, note that $\coker L$ is a quotient of $\coker L|_{Y}$. 

As $Y = \ker J$, from equation \ref{eqn:L} we get
\begin{equation}\label{eqn:LY}
(L - 115I)|_{Y}L|_{Y} = -(2\cdot 5^{3} \cdot 13)I.
\end{equation}
Take any pair of integer bases for $Y$ which put the matrix for $L|_{Y}$ into Smith normal form.  Follow a basis element $x$ through the composition of maps on the left side of equation \ref{eqn:LY}; we can see that the image is $-(2 \cdot 5^{3} \cdot 13)x$.  Hence the invariant factor of $L|_{Y}$ associated to the basis element $x$ must divide $2 \cdot 5^{3} \cdot 13$.  Said another way, the elementary divisors of $L|_{Y}$ can only be from among $\{2, 13, 5, 5^{2}, 5^{3}\}$, and so $\coker L|_{Y}$ has a cyclic decomposition of the form in the statement of the theorem.  The same must be true for its quotient $K(\Gamma)$.
\end{proof}

\begin{remark}
A \textit{bicycle} of $\Gamma$ is a subgraph for which every vertex has even degree and whose edges form an edge-cutset of $\Gamma$ (i.e., the deletion of the edges in the subgraph results in $\Gamma$ becoming disconnected).  The set of all bicycles of $\Gamma$ form a binary vector space with operation symmetric difference of edges.  The dimension of this vector space is equal to the number of invariant factors of $L$ that are even \cite[Lem. 14.15.3]{godsil-royle}.  Thus we have shown that $\Gamma$ has $2^{1728}$ bicycles--the maximum possible for the order of its critical group.
\end{remark}

In the next theorem we will flesh out a relationship between the integers $e_1, e_2, e_3$ and the $5$-rank of $L$, which we denote by $e_{0}$.  As $Syl_{5}(K(\Gamma))$ is the mystery here, it will be convenient to ignore all other primes than $5$.   We now briefly explain how to do this.

For a prime integer $p$, let $\Zp$ denote the ring of $p$-adic integers.  The ring $\Zp$ is a PID, so Smith normal form still makes sense for matrices with entries from $\Zp$; this of course encompasses all integer matrices.  When we view an integer matrix as having entries from the ring $\Zp$, the elementary divisors that survive the change of viewpoint are the powers of $p$.  The elementary divisor multiplicities can then be understood in terms of certain $\Zp$-modules attached to the matrix or map under consideration.

Let $\eta \colon \Zp^{n} \to \Zp^{m}$ be a homomorphism of free $\Zp$-modules of finite rank.  We get a descending chain of submodules of the domain
\[
\Zp^{n} = M_{0} \supseteq M_{1} \supseteq M_{2} \supseteq \cdots
\]
by defining 
\[
M_{i} = \left\{ x \in \Zp^{n} \, \vert \, \eta(x) \in p^{i}\Zp^{m} \right\}.
\]
That is, $M_{i}$ consists of the domain elements whose images under $\eta$ are divisible by $p^{i}$.

In a similar way, we can define
\[
N_{i} = \left\{ p^{-i} \eta(x) \, \vert \, x \in M_{i} \right\}.
\]
This gives us an ascending chain of modules in the codomain
\[
N_{0} \subseteq N_{1} \subseteq N_{2} \subseteq \cdots
\]
that will eventually stabilize to the purification of $\im \eta$ in $\Zp^{m}$.  For a submodule $R$ of the free $\Zp$-module $\Zp^{\ell}$, we define 
\[
\overline{R} = \left(R + p\Zp^{\ell}\right) / p\Zp^{\ell}.
\]
Note that $\overline{R}$ is a vector space over the finite field $\Fp = \Zp / p\Zp$.  We denote the field of fractions of $\Zp$ by $\Qp$.

\begin{lemma} \label{lem:dims}
Let $\eta \colon \Zp^{n} \to \Zp^{m}$ be a homomorphism of free $\Zp$-modules of finite rank.  Let $e_{i}$ denote the multiplicity of $p^{i}$ as an elementary divisor of $\eta$.  Then, for $i \geq 0$,
\[
\dim_{\Fp} \overline{M_{i}} = \dim_{\Fp} \overline{\ker(\eta)} + e_{i} + e_{i+1} + \cdots
\]
and
\[
\dim_{\Fp} \overline{N_{i}} = e_{0} + e_{1} + \cdots + e_{i}.
\]
\end{lemma}

\begin{proof} \hfil \\
Take a basis $\BB$ of the domain and a basis $\CC$ of the codomain for which the matrix of $\eta$ is in Smith normal form.  For $i \geq 0$, define the subset of $\BB$
\[
B_{i} = \{ x \in \BB \, \vert \, p^{i} \mbox{ divides } \eta(x), \mbox{ but } p^{i+1} \nmid \eta(x)\}.
\]
Then the basis $\BB$ is partitioned by the sets $\{B_{i}\}$ along with 
\[
D = \{ x \in \BB \, \vert \, \eta(x) = 0\}.
\]
In other words, we split $\BB$ up so that basis elements associated to the same invariant factor are grouped together.  Note that $B_{i}$ has cardinality $e_{i}$ and $D$ is a basis for $\ker(\eta)$.  A little thought reveals that a basis for $M_{i}$ is given by the set
\[
D \, \cup \, p^{i}B_{0} \, \cup \, p^{i-1}B_{1} \, \cup \cdots \cup \, pB_{i-1} \, \cup \, \left(\bigcup_{k \geq i} B_{k}\right).
\]
The nonzero elements of the $\Fp$-reduction of this set yields a basis of $\overline{M_{i}}$, and the first part of the lemma is proved.  By considering a similar partition of $\CC$ the second part of the lemma becomes clear as well.
\end{proof}

\begin{theorem}\label{thm:B}
Let $\Gamma$ be a Moore$(57,2)$ graph.  Let $e_{0}$ denote the rank of the Laplacian matrix of $\Gamma$ over a field of characteristic $5$.  Then either
\[
Syl_{5}(K(\Gamma)) \cong \left(\ZZ / 5 \ZZ\right)^{1520-e_{0}} \oplus \left(\ZZ / 5^{2} \ZZ\right)^{1732-e_{0}} \oplus \left(\ZZ / 5^{3} \ZZ\right)^{e_{0}-3}
\]
or
\[
Syl_{5}(K(\Gamma)) \cong \left(\ZZ / 5 \ZZ\right)^{1521-e_{0}} \oplus \left(\ZZ / 5^{2} \ZZ\right)^{1730-e_{0}} \oplus \left(\ZZ / 5^{3} \ZZ\right)^{e_{0}-2}.
\]
\end{theorem}

\begin{proof} \hfil \\
We view the Laplacian matrix $L$ of $\Gamma$ as a matrix over $\ZZ_{5}$.  For $\lambda$ an eigenvalue of $L$, let $V_{\lambda}$ denote the $\QQ_{5}$-eigenspace for $\lambda$.  One sees that $V_{65} \cap \ZZ_{5}^{V(\Gamma)} \subseteq N_{1}$, and so $\overline{V_{65} \cap \ZZ_{5}^{V(\Gamma)}} \subseteq \overline{N_{1}}$.  Since $V_{65} \cap \ZZ_{5}^{V(\Gamma)}$ is a direct summand of $\ZZ_{5}^{V(\Gamma)}$ (being the kernel of the endomorphism $L-65I$ of the $\ZZ_{5}$-lattice $\ZZ_{5}^{V(\Gamma)}$) with rank equal to the dimension of $V_{65}$ over $\QQ_{5}$, we have that $\dim_{\QQ_{5}} V_{65} = \dim_{\FF_{5}} \overline{V_{65} \cap \ZZ_{5}^{V(\Gamma)}}$.  Applying Lemma \ref{lem:dims},
\begin{align}
1520 &= \dim_{\QQ_{5}} V_{65} \label{ineq:1} \\
&= \dim_{\FF_{5}} \overline{V_{65} \cap \ZZ_{5}^{V(\Gamma)}} \nonumber \\
&\leq \dim_{\FF_{5}} \overline{N_{1}} \nonumber\\
&= e_{0} + e_{1}.\nonumber
\end{align}
By a similar argument, $V_{50} \cap \ZZ_{5}^{V(\Gamma)} \subseteq M_{2}$ and Lemma \ref{lem:dims} implies that 
\begin{align}
1729 &= \dim_{\QQ_{5}} V_{50} \label{ineq:2} \\
&= \dim_{\FF_{5}} \overline{V_{50} \cap \ZZ_{5}^{V(\Gamma)}} \nonumber \\
&\leq \dim_{\FF_{5}} \overline{M_{2}} \nonumber\\
&= 1 + e_{2} + e_{3}.\nonumber
\end{align}
Note that $\ker L$ is spanned by the all-one vector $\allone$, which explains the $1$ appearing in the right hand side of the above inequality.

Now consider carefully these two inequalities \ref{ineq:1} and \ref{ineq:2}:
\begin{align*}
1520 &\leq e_{0} + e_{1}\\
1729 &\leq 1 + e_{2} + e_{3}.
\end{align*}
The sum of the left hand sides is $1520+1729=3249$, while the sum of the right hand sides is $e_{0}+e_{1}+e_{2}+e_{3}+1=3250$.  There are exactly two ways in which this can be:\\

\begin{enumerate} 
\item[\textbf{Case 1:}]  $\quad 1520 = e_{0} + e_{1} \mbox{ and } 1729 = e_{2} + e_{3}$. \label{A}\\
\item[\textbf{Case 2:}]  $\quad 1521 = e_{0} + e_{1} \mbox{ and } 1728 = e_{2} + e_{3}$.
\end{enumerate}
\vspace{.5cm}
There is another equation that applies to all cases.  Since 
\[ 
|Syl_{5}(K(\Gamma))| = 5^{4975},
\]
 we have
\begin{equation} \label{val}
4975 = e_{1} + 2e_{2} + 3e_{3}.
\end{equation}

Taking equation \ref{val} with the two equations of Case 1, we are seeking nonnegative integer solutions to the system
\begin{align*}
e_{0} + e_{1}  &= 1520 \\
e_{2} + e_{3} &= 1729\\
e_{1} + 2e_{2} + 3e_{3} &= 4975.
\end{align*}
This is easily done by hand.  Choosing, say, $e_{3}$ to be free we get:
\begin{itemize}
\item $e_{3} = t$
\item $e_{2} = 1729 - t$
\item $e_{1} = 1517 - t$
\item $e_{0} = 3 + t$.
\end{itemize}
Writing each unknown in terms of the $5$-rank $e_{0}$ instead gives us the first isomorphism in the statement of the theorem.

In Case 2, the system becomes
\begin{align*}
e_{0} + e_{1}  &= 1521 \\
e_{2} + e_{3} &= 1728\\
e_{1} + 2e_{2} + 3e_{3} &= 4975.
\end{align*}
The solutions may be written 
\begin{itemize}
\item $e_{3} = t$
\item $e_{2} = 1728 - t$
\item $e_{1} = 1519 - t$
\item $e_{0} = 2 + t$.
\end{itemize}
If we instead take $e_{0}$ to be free we get multiplicities as in the second isomorphism of the theorem.

\end{proof}

\begin{remark}
The author has thus far been unable to obtain strong bounds on the possible $5$-rank of $L$.  The ambitious reader is directed to \cite{b-ve}; there the authors compute the relevant $p$-ranks of the Petersen graph and the Hoffman-Singleton graph.  Knowledge of specific adjacencies and constructions within the graphs are used.
\end{remark}

\section{Acknowledgements}
The author thanks an anonymous referee for helpful comments.  This work was supported by James Madison University's Tickle Fund.

\end{document}